\documentclass[12pt]{article}
\usepackage{amsmath, amsthm, amssymb,latexsym,enumerate}
\usepackage{graphicx}
\usepackage{hyperref}
\usepackage{xcolor}
\usepackage{todonotes}
\usepackage{cite}
\usepackage{dsfont}

\newtheorem{theorem}{Theorem}[section]

\newtheorem{claim}[theorem]{Claim}
\newtheorem{conjecture}[theorem]{Conjecture}
\newtheorem{question}[theorem]{Question}

\numberwithin{equation}{section}

\title{On independent domination of regular graphs} 

\author{Eun-Kyung Cho\thanks{
Department of Mathematics, Hankuk University of Foreign Studies, Yongin-si, Gyeonggi-do, Republic of Korea.
 \texttt{ekcho2020@gmail.com}
}
\and Ilkyoo Choi\thanks{
Department of Mathematics, Hankuk University of Foreign Studies, Yongin-si, Gyeonggi-do, Republic of Korea.
\texttt{ilkyoo@hufs.ac.kr}
}
\and Boram Park\thanks{
Department of Mathematics, Ajou University, Suwon-si, Gyeonggi-do, Republic of Korea.
\texttt{borampark@ajou.ac.kr}
}}
\date\today

\begin{document}
 
\maketitle

\begin{abstract} 
Given a graph $G$, a {\it dominating set} of $G$ is a set $S$ of vertices such that each vertex not in $S$ has a neighbor in $S$.  
The {\it domination number} of $G$, denoted $\gamma(G)$, is the minimum size of a dominating set of $G$. 
The {\it independent domination number} of $G$, denoted $i(G)$, is the minimum size of a dominating set of $G$ that is also independent. 
Note that every graph has an independent dominating set, as a maximal independent set is equivalent to an independent dominating set. 

Let $G$ be a connected $k$-regular graph that is not $K_{k, k}$ where $k\geq 4$. 
Generalizing a result by Lam, Shiu, and Sun, we prove that  $i(G)\le \frac{k-1}{2k-1}|V(G)|$, which is tight for $k = 4$.
This answers a question by Goddard et al.
in the affirmative. 
We also show that $\frac{i(G)}{\gamma(G)} \le \frac{k^3-3k^2+2}{2k^2-6k+2}$, strengthening upon a result of Knor, \v Skrekovski, and Tepeh.
In addition, we prove that a graph $G'$ with maximum degree at most $4$ satisfies $i(G') \le \frac{5}{9}|V(G')|$, which is also tight. 
\end{abstract}

\section{Introduction}

Let $G$ be a finite simple graph.
Let $V(G)$ and $E(G)$ denote the vertex set and the edge set, respectively, of $G$. 
A {\it dominating set} of $G$ is a subset $S$ of $V(G)$ such that each vertex not in $S$ has a neighbor in $S$. 
The {\it domination number} of $G$, denoted $\gamma(G)$, is the minimum size of a dominating set of $G$.
Domination is an extensively studied classic topic in graph theory, to the point that 
there are several books focused solely on domination, see~\cite{book1,book2,book3,book4}.

A dominating set that is also an independent set is an {\it independent dominating set}.
The {\it independent domination number} of $G$, denoted by $i(G)$, is the minimum size of an independent dominating set of $G$. 
Note that every graph has an independent dominating set, as a maximal independent set is equivalent to an independent dominating set. 
This concept appears in the literature as early as 1962 by Berge~\cite{berge1962theory} and Ore~\cite{ore1962theory}.
For a survey regarding independent domination, see~\cite{goddard2013independent}.

We focus on finding the maximum (constant) ratio of the independent domination number and the number of vertices for regular graphs. 
Surprisingly, not much is known for $k$-regular graphs when $k\geq4$. 
We are also interested in the class of graphs with bounded maximum degree. 
We first lay out  related  literature for the independent domination number of regular graphs.

For a connected $k$-regular graph $G$ where $k\geq 1$, Rosenfeld~\cite{rosenfeld1964independent} showed that $i(G)\leq \frac{|V(G)|}{2}$, which is tight only for the balanced complete bipartite graph $K_{k, k}$. 
We are interested in lowering the upper bound on the independent domination number when the balanced complete bipartite graph is excluded. 
Note that there are no connected $1$-regular graph when $K_{1, 1}$ is excluded.
When $k=2$, so $G$ is a cycle, one can easily calculate that $i(G)\leq \frac{3}{7}|V(G)|$ holds except for the $4$-cycle, which is $K_{2, 2}$.
Extending this pattern, Lam, Shiu, and Sun \cite{lam1999independent} showed the below result for {\it cubic} graphs, which are $3$-regular graphs:

\begin{theorem}[\cite{lam1999independent}]\label{thm:cubic}
If $G$ is a cubic graph on at least $8$ vertices, then   $i(G)\leq \frac{2}{5}|V(G)|$, and the bound is tight by $C_5\square K_2$. See the left graph in Figure~\ref{fig:prism}.
\end{theorem}

Since the only cubic graph on at most 6 vertices that does not satisfy the above theorem is $K_{3, 3}$, one can reinterpret the above theorem as the following: if $G$ is a cubic graph that is not $K_{3, 3}$, then $i(G)\leq \frac{2}{5}|V(G)|$, which is tight for $C_5\square K_2$. 

\begin{figure}[h!]
\centering
\includegraphics[width=3cm, page  = 1]{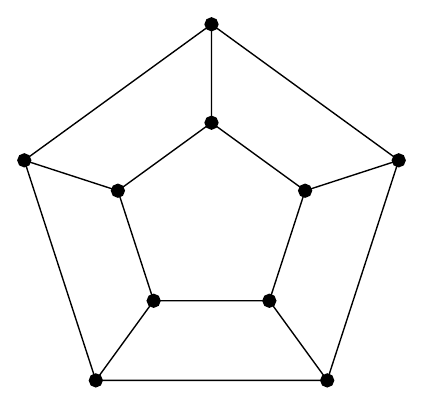}
\qquad \qquad
\includegraphics[width=3cm, page =2]{fig_ind_dom_all.pdf}
\qquad \quad
\includegraphics[width=4cm, page = 3]{fig_ind_dom_all.pdf}
\caption{The graph $C_5\square K_2$, the $4$-regular expansion of a $7$-cycle, and the graph $H(q,p)$}\label{fig:prism}
\end{figure}

As $K_{3, 3}$ is the only graph where equality holds in Rosenfeld's upper bound on the independent domination number for cubic graphs, Goddard and Henning~\cite{goddard2013independent} conjectured that there is only one graph where the upper bound in Theorem~\ref{thm:cubic} is tight. 
Namely, they conjectured that if $G$ is a connected cubic graph that is neither $K_{3, 3}$ nor $C_5 \square K_2$, then $i(G) \le \frac{3}{8}|V(G)|$.
 This conjecture is still open, and as a partial result, Dorbec et al.~\cite{dorbec2015independent} 
showed that the conjecture holds if in addition $G$ does not have a subgraph isomorphic to $K_{2,3}$.
For other conjectures and partial results regarding the independent domination number of subclasses of cubic graphs,
see~\cite{goddard2013independent,goddard2012independent,abrishami2018independent,duckworth2006independent}.

Unlike cubic graphs, little was known for $k$-regular graphs where $k\ge 4$. 
Let $H'$ be the $4$-regular expansion of a $7$-cycle, see the middle graph in Figure~\ref{fig:prism}. 
Goddard et al.~\cite{goddard2012independent} observed that $H'$ satisfies $i(H') = \frac{3}{7}|V(H')|$, and asked the following question:

\begin{question}[\cite{goddard2012independent}]\label{question:4reg} 
If $G$ is a connected $4$-regular graph that is not $K_{4,4}$, then does $i(G)\leq \frac{3}{7}|V(G)|$ hold?
\end{question}

Our first result answers the above question in affirmative.
We actually prove a theorem that applies to all $k$-regular graphs where $k\geq 3$, so our result also encompasses Theorem~\ref{thm:cubic}. 

\begin{theorem}\label{mainthm:kreg}
For $k \ge 3$, if $G$ is a connected $k$-regular graph that is not $K_{k, k}$, then $i(G)\le \frac{k-1}{2k-1}|V(G)|$.
\end{theorem}

To our knowledge, this is the best upper bound on the independent domination number for general $k$. 
Note that the left and middle graphs in Figure~\ref{fig:prism} demonstrate that the bound in Theorem~\ref{mainthm:kreg} is tight for $k \in \{3,4\}$.
Whether the bound is tight or not for $k \ge 5$ is unknown as we were unable to construct such examples. 

\medskip

We turn our attention to the ratio of the independent domination number and the domination number for connected regular graphs.
Note that the ratio can be arbitrarily large, as it is for the complete bipartite graph, so we seek to obtain a bound that depends on the regularity. 
Note that the independent domination number and the domination number does not differ for $k$-regular graphs when $k\leq 2$.  

For cubic graphs,
Goddard et al.~\cite{goddard2012independent}
proved that if $G$ is a connected cubic graph, then $\frac{i(G)}{\gamma(G)}\leq \frac{3}{2}$, and the bound is tight if and only if $G=K_{3,3}$. 
Southey and Henning~\cite{southey2013domination} extended the result by showing that $\frac{i(G)}{\gamma(G)}\leq \frac{4}{3}$ when $G\neq K_{3,3}$, and the bound is tight if and only if $G=C_5\square K_2$. 
O and West \cite{suil2016cubic} constructed  an infinite family of connected cubic graphs $G$ such that $\frac{i(G)}{\gamma(G)}=\frac{5}{4}$, and asked if there are only finitely many exceptions to the statement that a connected cubic graph $G$ satisfies $\frac{i(G)}{\gamma(G)} \le \frac{5}{4}$.

There was recent activity in investigating the ratio under consideration for $k$-regular graphs where $k\geq 4$. 
Babikir and Henning~\cite{babikir2020domination} showed that the statement of Goddard et al. in the previous paragraph also holds for $k$-regular graphs where $k\in\{4,5,6\}$.
Very recently, Knor, \v Skrekovski, and Tepeh~\cite{knor2020domination} generalized the result to all $k\geq 3$; namely, it is now known that 
for all $k \ge 2$, if $G$ is a connected $k$-regular graph, then $\frac{i(G)}{\gamma(G)}\leq \frac{k}{2}$, and the bound is tight if and only if $G=K_{k,k}$.

It is natural to ask if there exists a better bound than $\frac{k}{2}$ when $K_{k,k}$ is excluded, as it is the case for cubic graphs. 
Using Theorem~\ref{mainthm:kreg}, we are able to provide a better upper bound than $\frac{k}{2}$. 

\begin{theorem}\label{thm:ratio:kreg}
For $k \ge 4$, if $G$ is a connected $k$-regular graph that is not $K_{k,k}$, then $\frac{i(G)}{\gamma(G)} \le \frac{k^3-3k^2+2}{2k^2-6k+2}$. 
\end{theorem}

Note that $\frac{k^3-3k^2+2}{2k^2-6k+2} < \frac{k}{2}$ for all $k \ge 4$.
In particular, the bound becomes $\frac{9}{5}$ when $k=4$. 
To our knowledge, this is the first partial answer to the following question asked in~\cite{goddard2012independent}: does $\frac{i(G)}{\gamma(G)}\leq \frac{3}{2}$ hold for a connected $4$-regular graph $G$ that is not $K_{4, 4}$?
If the aforementioned question is true, then it is tight by the $4$-regular expansion of a $7$-cycle (and also an $8$-cycle).

\bigskip

We now switch gears and consider the family of graphs with bounded maximum degree. 
Since an {\it isolated vertex}, which is a vertex of degree $0$, must be part of every independent dominating set, we consider the class of graphs without isolated vertices; we call these graphs {\it isolate-free}.
Let $H(q, p)$ be the graph obtained by attaching $p$ pendent vertices to every vertex of a complete graph on $q$ vertices. 
See the right graph in Figure~\ref{fig:prism}.

Akbari et al.~\cite{akbari2021independent} proved that if $G$ is an isolate-free graph with maximum degree at most $3$, then $i(G)\le\frac{|V(G)|}{2}$, and they also characterized all graphs where equality holds. 
In this vein, we extend their result by proving a sharp upper bound on the independent domination number for isolate-free graphs with  maximum degree at most $4$. 

\begin{theorem}\label{thm:idset4}
If $G$ is an isolate-free graph with maximum degree at most $4$, then $i(G) \le \frac{5}{9}|V(G)|$, and equality holds for $H(3, 2)$. 
\end{theorem}

The above theorem is tight, as demonstrated by the graph $H(3, 2)$. 
We actually think the family of graphs $H\left(\left\lfloor\frac{D}{2}\right\rfloor+1, \left\lceil\frac{D}{2}\right\rceil\right)$ has the maximum independent domination number among  isolate-free graphs with maximum degree at most $D$, so we put forth the following conjecture:

\begin{conjecture}\label{con:idset-general}
If $G$ is an isolate-free graph with maximum degree $D\geq1$, then 
\begin{eqnarray*}
i(G) \le \begin{cases}
\frac{D^2+4}{(D+2)^2}|V(G)| & \text{ if $D$ is even,} \\
\frac{D^2+3}{(D+1)(D+3)}|V(G)| & \text{ if $D$ is odd.}
\end{cases}
\end{eqnarray*}
\end{conjecture}

If Conjecture~\ref{con:idset-general} is true, then $H\left(\left\lfloor\frac{D}{2}\right\rfloor+1, \left\lceil\frac{D}{2}\right\rceil\right)$ demonstrates that the bound is tight. 
Note that one can easily check that Conjecture~\ref{con:idset-general} is true for $D\leq 2$, and we remark that Conjecture~\ref{con:idset-general} is true for $D=3$ and $D=4$ by the result of Akbari et al.~\cite{akbari2021independent} and by Theorem~\ref{thm:idset4}, respectively. 
In addition, we checked that the conjecture holds for $D \in \{5,6,7,8\}$, but we decided to not include the proofs as it mainly consists of tedious case checking.
For results regarding the ratio of the independent domination number and the domination number for graphs with bounded maximum degree, see 
\cite{rad2013note,furuya2014ratio}.

In Section~\ref{sec:mainthm:kreg}, we prove Theorem~\ref{mainthm:kreg}.
The proof essentially boils down to one (implicit) inequality, which was inspired by an idea  in~\cite{lam1999independent}.
However, unlike their proof, we use discharging to prove that the inequality holds. 
Using Theorem~\ref{mainthm:kreg} and an idea in~\cite{southey2013domination}, we prove  Theorem~\ref{thm:ratio:kreg} in Section~\ref{sec:ratio}.
In Section~\ref{sec:idset4}, we prove a statement slightly stronger than Theorem ~\ref{thm:idset4}, where the proof adopts the approach of~\cite{dorbec2015independent}. 

We end the introduction with some notation and terminology used in this paper. 
Given a graph $G$, let $\Delta(G)$ and $\delta(G)$ denote the maximum degree and the minimum degree, respectively, of $G$.
For each $X \subseteq V(G)$, let  $G-X$  denote the subgraph of $G$ induced by $V(G) \setminus X$.
Let $n_0(G)$ denote the number of isolated vertices of $G$.
For each vertex $v \in V(G)$,  the degree of $v$ in $G$ is denoted by $\deg_G(v)$. 
For each $v \in V(G)$, let $N_G(v)$ denote the set of neighbors of $v$, and let 
$N_G[v]=N_G(v)\cup \{v\}$.
For each $X \subseteq V(G)$, let $N_G(X) = \bigcup_{v \in X} N_G(v)$ and $N_G[X] = N_G(X) \cup X$.
For a vertex $v$ and $X\subseteq V(G)$, an {\it $X$-neighbor} of $v$ is a neighbor of $v$ in $X$. 
A \textit{minimum dominating set} of $G$ is a dominating set of $G$ with size $\gamma(G)$, and a \textit{minimum independent dominating set} of $G$ is an independent dominating set of $G$ with size $i(G)$.

\section{Independent domination of regular graphs}\label{sec:mainthm:kreg}

In this section, we prove Theorem~\ref{mainthm:kreg}. 
For $I\subseteq V(G)$, let $G_I$ be the spanning bipartite graph obtained from $G$ by deleting the edges joining two vertices in $I$ and the edges joining two vertices in $V(G)\setminus I$. 
For brevity, denote $N_{G_I}(v)$ and $\deg_{G_I}(v)$ by $N_I(v)$ and $\deg_I(v)$, respectively.
 
For $k = 3$, the theorem holds by Lam, Shiu, and Sun~\cite{lam1999independent}. 
Fix $k \ge 4$.
Let $G \neq K_{k,k}$ be a connected $k$-regular graph. 
Choose a minimum independent dominating set $I$ of $G$ that
\begin{itemize}
 \item[(1)] minimizes the number of subgraphs in $G_I$ isomorphic to $K_{k-1,k}$, and 
 \item[(2)] maximizes the number of pendent vertices $v$ of $G_I$ such that  the $I$-neighbor $w$ of $v$ has a neighbor $x\in V(G)\setminus I$ satisfying 
 $\deg_I(x)=k$.
\end{itemize}
Since $I$ is an independent set of $G$, all vertices in $I$ have degree $k$ in $G_I$. 
Let $J=V(G)\setminus I$ and  $J_i=\{v\in J\mid |N_I(v)|=i\}$ for each $i\in [k]$.
Note that $J_1,\ldots,J_k$ form a partition of $J$, since $I$ is a dominating set of a $k$-regular graph $G$. 
For two integers $s, t\in\{1, \ldots,  k\}$, let $J_{[s,t]}$ denote $\bigcup_{i \in \{s,\ldots,t\}} J_i$.

For each $v\in J_k$, let $X(v)$ be the set of vertices in $J\setminus\{v\}$ whose $I$-neighbors are in $N_G(v)$, namely, 
$$X(v)=\{ w\in J\setminus\{v\} \mid N_I(w) \subseteq N_G(v)\}.$$ 

For each $v\in J_{[1,k-1]}$, let $Y(v)$ be the set of vertices $w\in J_k$ such that $v$ belongs to $X(w)$, namely,
$$Y(v)=\{ w\in J_k\mid v\in X(w)\}.$$ 
For $u\in {J_{[1,{k-1}]}}$ and $v\in J_k$, note that $u\in X(v)$ if and only if $v\in Y(u)$.
Using this terminology, (2) can be rephrased as the following:

\smallskip

 (2) maximizes the number of vertices $v\in J_1$ such that $Y(v)\neq\emptyset$.

\begin{claim}\label{claim-basic-k}
The following holds:
\begin{itemize}
    \item[\rm(i)] If $v\in J_{[1,k-1]}$, then $|Y(v)|\le k-1$. 
    \item[\rm (ii)] If $v \in J_i$ for $i \in [k]$, then there are at least $i$ vertices $w$ such that $N_I(w) \subseteq N_I(v)$.
    In particular, if $v\in J_k$, then $|X(v)|\ge k-1$.
\end{itemize}
\end{claim}

\begin{proof}
Since $G$ is $k$-regular, (i) follows from the definition of $Y(v)$.

To show (ii), let $v\in J_i$ for some $i \in [k]$. 
Note that  $I'=(I\setminus N_G(v))\cup \{w \mid N_I(w) \subseteq N_I(v)\}$ contains an independent dominating set of $G$. 
To be precise, if $X'$ is a maximal independent set of the subgraph induced by $\{w \mid N_I(w) \subseteq N_I(v)\}$ such that $v\in X'$,
then $(I\setminus N_G(v))\cup X'$ is  an independent dominating set of $G$.
If $|\{w \mid N_I(w) \subseteq N_I(v)\}|\le i-1$, then $|X'|\le i-1$, so $I'$ is a smaller independent dominating set than $I$, which is a contradiction to the choice of $I$. Hence,  $|\{w \mid N_I(w) \subseteq N_I(v)\}|\ge i$, so (ii) holds.
\end{proof}

\begin{claim}\label{claim-temp-k}
For $v \in J_k$, if $X(v)$ contains at most $k-3$ vertices in $J_1$  and there are $k-1$ distinct vertices $v_1,\ldots,v_{k-1} \in X(v) \cap J_{k-1}$ such that $N_I(v_1)=\cdots=N_I(v_{k-1})$, then $v_1,\ldots,v_{k-1}$ have a common $J_1$-neighbor $u$ such that $|Y(u)| \le k-3$ and $\cup_{w\in J_{k-1}\cap N_G(u)}Y(w)=\{v\}$.
\end{claim}

\begin{proof}
Assume $N_I(v_1)=\cdots=N_I(v_{k-1}) = \{v'_1,\ldots,v'_{k-1}\}$ and let $N_G(v) = \{w',v'_1,\ldots,v'_{k-1}\}$.
Now, $I' = (I\setminus\{v'_1,\ldots,v'_{k-1}\}) \cup \{v_1,\ldots,v_{k-1}\}$ contains an independent dominating set of $G$.
See Figure~\ref{fig:37-3-k} for an illustration.

\begin{figure}[h!]
\centering
\includegraphics[height=3.5cm, page = 4]{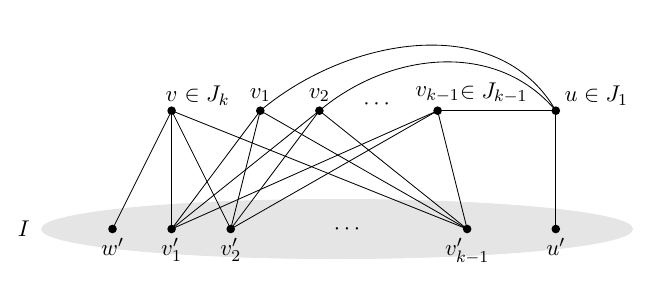}
\caption{An illustration for the proof of Claim~\ref{claim-temp-k} }\label{fig:37-3-k}
\end{figure}

Since $I$ and $I'$ have the same size but $I$ was chosen over $I'$, according to condition (1) of the choice of $I$, the number of subgraphs of $G_{I'}$ isomorphic to $K_{k-1,k}$ in $G_{I'}$ is at least that in $G_{I}$. 
Since $v,v_1,\ldots,v_{k-1},v'_1,\ldots,v'_{k-1}$ form a graph isomorphic to $K_{k-1,k}$ in $G_I$, it follows that $v_1,\ldots,v_{k-1}$ are part of a subgraph of $G_{I'}$ isomorphic to $K_{k-1,k}$.
Thus $v_1,\ldots,v_{k-1}$ have a common neighbor $u$ in $J$.
Since $Y(v_i) = \{v\}$ for all $i \in [k-1]$, we have $\cup_{w\in J_{k-1}\cap N_G(u)}Y(w)=\{v\}$.
Since $I$ must dominate $u$, we know  $u\in J_1$, and let  $u'$ be the $I$-neighbor of $u$.

Now, $I'=(I\setminus\{w',v'_1,\ldots,v'_{k-1},u'\})\cup \{v\} \cup (X(v) \cap J_1)\cup (N_G(u')\setminus Y(u))$ contains an independent dominating set of $G$ since $u$ and $v$ dominate $w', v'_1, \ldots, v'_{k-1}, u'$, and every vertex that is not dominated by $I\setminus\{w',v'_1,\ldots,v'_{k-1},u'\}$ is dominated by $\{v\} \cup (X(v)\cap J_1 )\cup (N_G(u')\setminus Y(u)$).
Note that $|I'|=|I|-(k+1)+1+|X(v)\cap J_1|+k-|Y(u)|\leq |I|+k-3-|Y(u)|$.
If $|Y(u)|\ge k-2$, then $|I'|<|I|$, which is a contradiction to the choice of $I$. 
Hence, $|Y(u)|\le k-3$.
\end{proof}

\begin{claim}\label{claim-basic2-k}
For $v\in J_k$, suppose that $X(v)\cap J_{[2,k-2]}=\emptyset$. 
If either $X(v)\cap J_1=\emptyset$ or $X(v)\cap J_1=\{w\}$ where $|Y(w)|=k-1$, then $X(v)\setminus J_1$ consists of $k-1$ distinct vertices 
in $J_{k-1}$ with a common $J_1$-neighbor $u$ such that $|Y(u)| \le k-3$ and $\cup_{w\in J_{k-1}\cap N_G(u)}Y(w)=\{v\}$.
\end{claim}

\begin{proof} 
Let $w', v'_1, \ldots, v'_{k-1}$ be the $I$-neighbors of $v$.

We first consider the case when $X(v)\cap J_1=\emptyset$, so $X(v)\subseteq J_{[k-1,k]}$. 
Since $G\neq K_{k,k}$, $X(v)$ contains a vertex $w\in J_{k-1}$. 
Assume $w'$ is not adjacent to $w$. 
Let $S$ be the set of vertices $u\in X(v)\cap J_{k-1}$ such that $N_I(w)=N_I(u)$.
Then $|S| \ge k-1$ by Claim~\ref{claim-basic-k}(ii).
Since each vertex in $N_G(v)$ has degree $k$ and $X(v)\cap J_1=\emptyset$, every vertex in $X(v)$ must be in $S$, so $|X(v)|=|S|=k-1$. 
By Claim~\ref{claim-temp-k}, the vertices in $X(v)$ have a common $J_1$-neighbor $u$ such that $|Y(u)|\le k-3$ and $\cup_{w\in J_{k-1}\cap N_G(u)}Y(w)=\{v\}$.

Now we consider the case when $X(v)\cap J_1=\{w\}$ where $|Y(w)|=k-1$. 
Assume $w'$ is the $I$-neighbor of $w$. 
Suppose that $X(v)\setminus\{w\}\subseteq J_{k-1}$. 
Since $|Y(w)|=k-1$, each $z \in N_G(w') \setminus \{w\}$ 
is in $J_k$, so $z\not\in X(v)$. 
Thus, a vertex in $X(v)\setminus\{w\}$ cannot be adjacent to $w'$, so the $I$-neighbors of each vertex in $X(v)\setminus\{w\}$ are $v'_1,\ldots,v'_{k-1}$.
By applying Claim~\ref{claim-basic-k}(ii) to a vertex in $X(v) \cap J_{k-1}$, we know $|X(v)\setminus\{w\} |\ge k-1$.
Since $G$ is $k$-regular, we obtain $|X(v)\setminus\{w\}| = k-1$. 
By Claim~\ref{claim-temp-k}, the vertices in $X(v) \setminus \{w\}$ have a common $J_1$-neighbor $u$ such that  $|Y(u)| \le k-3$ and $\cup_{w\in J_{k-1}\cap N_G(u)}Y(w)=\{v\}$. 
We will complete the proof by showing that $X(v)\setminus\{w\} \subseteq J_{k-1}$ always holds. 

Suppose to the contrary that there is a vertex $v_1\in X(v) \cap J_k$.
Since $X(v) \cap J_1 = \{w\}$ and $X(v) \cap J_{[2,k-2]} = \emptyset$, every vertex in $X(v)\setminus\{w,v_1\}$ is in $J_{[k-1,k]}$. 
Note that $X(v)\setminus\{w,v_1\}\neq \emptyset$ by Claim~\ref{claim-basic-k}(ii) since $k\geq 4$.
If $(X(v)\setminus\{w,v_1\}) \cap J_{k-1} \neq \emptyset$, then $|(X(v)\setminus\{w,v_1\}) \cap J_{k-1}| \ge k-1$ by Claim~\ref{claim-basic-k}(ii).
However, by counting the number of edges between $X(v)\setminus\{w, v_1\}$ and $N_G(v)$, we obtain $|(X(v)\setminus\{w,v_1\}) \cap J_{k-1}|\le k-2$, which is a contradiction.
See Figure~\ref{fig:37-1-k}.

\begin{figure}[h!]
\centering
\includegraphics[height=3.5cm, page = 5]{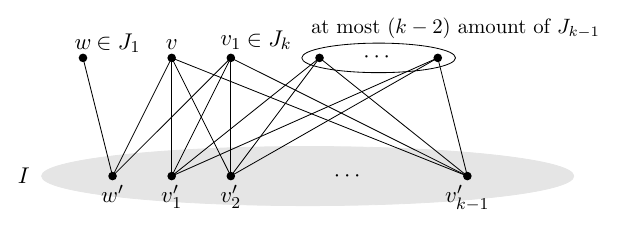}
\caption{An illustration when $(X(v)\setminus\{w,v_1\}) \cap J_{k-1} \neq \emptyset$}\label{fig:37-1-k}
\end{figure}

Thus, we may assume that every vertex in $X(v) \setminus \{w\}$ is in $J_k$.
By Claim~\ref{claim-basic-k}(ii), $|X(v)|\geq k-1$, which further implies that there are exactly $k-2$ vertices $v_1,\ldots,v_{k-2}$ in $X(v) \cap J_k$ since $G$ is $k$-regular. 
In particular, $X(v)=\{w,v_1,\ldots, v_{k-2}\}$. 
See Figure~\ref{fig:37-2-k}.
Note that $I^*=(I\setminus\{w'\})\cup \{w\}$ is an independent dominating set of $G$ with the same size as $I$. 
Let $J^*=V(G)\setminus I^*$ and  $J^*_i=\{v\in J^*\mid |N_{I^*}(v)|=i\}$ for each $i\in [k]$. 
Define $X^*(v)$ and $Y^*(v)$ analogously.

\begin{figure}[h!]
\centering
\includegraphics[height=3.5cm, page = 6]{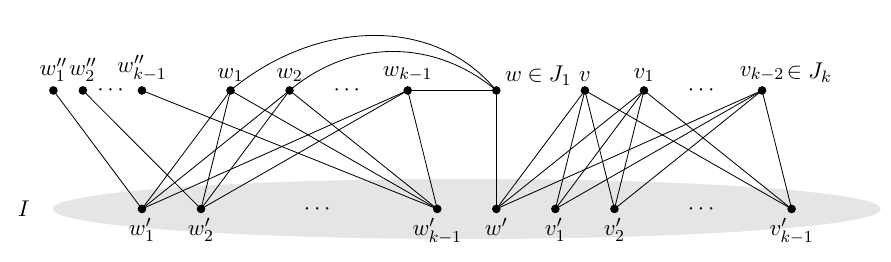}
\caption{An illustration when $X(v)\setminus\{w\}$ is a subset of $J_k$}\label{fig:37-2-k}
\end{figure}

Note that $v,v_1,\ldots,v_{k-2},v'_1,\ldots,v'_{k-1},w'$ form a graph isomorphic to $K_{k-1,k}$ in $G_I$, but not in $G_{I^*}$. 
By condition (1) of the choice of $I$, it follows that $w$ is a vertex of some subgraph $H$ isomorphic to $K_{k-1,k}$ in $G_{I^*}$; 
let the partite sets of $H$ be $\{w_1,\ldots,w_{k-1}\}$ and $\{w'_1,\ldots,w'_{k-1},w\}$. 
So, the number of subgraphs isomorphic to $K_{k-1,k}$ in $G_{I^*}$  is equal to that of $G_I$.
We will reach a contradiction by showing that the number of vertices $v\in J^*_1$ such that $Y^*(v)\neq\emptyset$ in $G_{I^*}$ is greater than that in $G_I$. 

For each $i\in [k-1]$, let $w''_i$ be the neighbor of $w'_i$ not in $\{w_1,\ldots,w_{k-1}\}$.
If $w''_i\not\in J_1$ for each $i\in [k-1]$, then $(I\setminus\{w'_1,\ldots,w'_{k-1},w',v'_1,\ldots,v'_{k-1}\}) \cup \{w_1,\ldots,w_{k-1},v,v_1,\ldots,v_{k-2}\}$ is a smaller independent dominating set of $G$ than $I$, which is a contradiction.
Thus, we may assume that $w''_1\in J_1$ (and therefore $w''_1\in J^*_1)$ and $N_I(w''_1)=\{w'_1\}$ (and also $N_{I^*}(w''_1)=\{w'_1\}$). 
Clearly, $Y(w''_1)=\emptyset$, but $Y^*(w''_1)=\{w_1,\ldots,w_{k-1}\}\neq \emptyset$.
Moreover, note that $w'\in J^*_1$ and $Y^*(w')=\{w_1,\ldots,w_{k-1}\}$.  
This is a contradiction to condition (2) of the choice of $I$, which completes the proof.
\end{proof}

Now, suppose to the contrary that  $|I|=i(G)> \frac{k-1}{2k-1}|V(G)|$, which implies $(2k-1)|I|-(k-1)|V(G)|>0$. 
For each vertex $v$, define the {\it initial charge} $\mu(v)$ of each vertex $v$ to be  \[ \mu(v)=\begin{cases} 
k & \text{if }v\in I\\
1-k & \text{if }v\in J.
\end{cases}\]
The sum of the initial charge is $k|I|+(1-k)(|V(G)|-|I|)=(2k-1)|I|-(k-1)|V(G)| >0$.

We distribute the initial charge according to the following {\it discharging rules}, which are designed so that the total charge is preserved, to obtain the {\it final charge} $\mu^*(v)$ at each vertex $v$.
We obtain a contradiction by showing that the sum of the final charge is non-positive, by proving that the final charge of each vertex is non-positive.

\begin{enumerate}[{\bf[R1]}]
\item\label{r1} 
Every vertex in $J$ sends $-1$ to each $I$-neighbor.
\item\label{r2} 
For $i \in \{2,\ldots,k-2\}$, every vertex $w\in J_i$ with $Y(w)\neq \emptyset$ sends $-\frac{k-1-i}{|Y(w)|}$ to each vertex in $Y(w)$.
\item\label{r3} For every vertex $w\in J_{1}$,
\begin{itemize}
    \item[{\bf[R3-1]}] if $|Y(w)|=k-1$, 
then $w$ sends $-\frac{k-2}{k-1}$ to each vertex in $Y(w)$.
\item[{\bf[R3-2]}] if $|Y(w)|\le k-2$, 
then $w$ sends $-1$ to each vertex in $Y(w)$.
\item[{\bf[R3-3]}] if $|Y(w)|\le k-3$ and $\displaystyle\bigcup_{u\in J_{k-1}\cap N_G(w)}Y(u)=\{x\}$, then $w$ sends $-1$ to the vertex $x$.
\end{itemize}
\end{enumerate}

\begin{figure}[h!]
\centering
\includegraphics[width=18cm, page = 7]{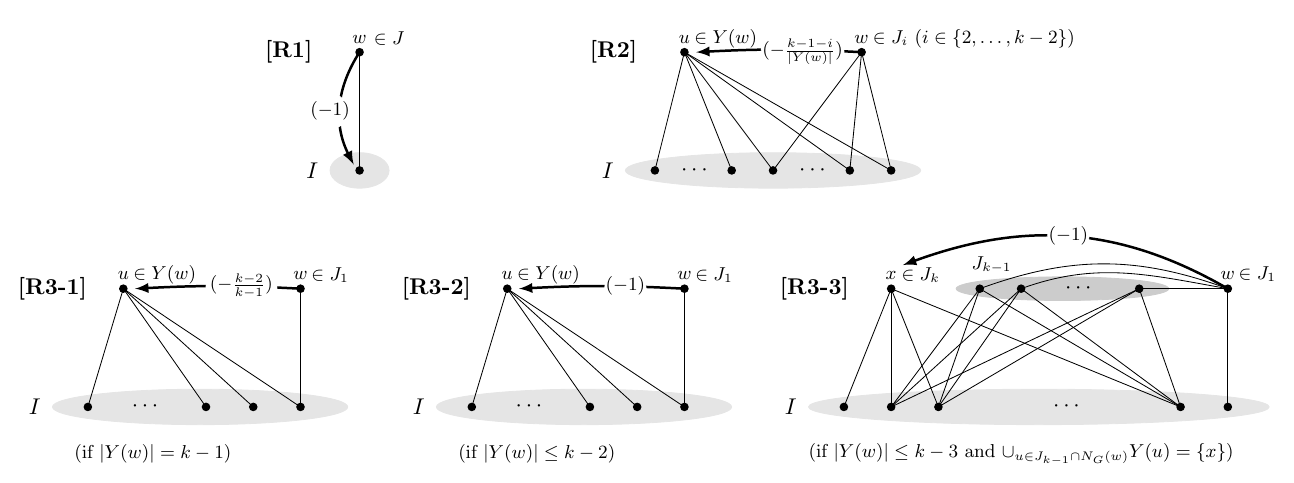}
\caption{Illustrations for the discharging rules}\label{fig:rules}
\end{figure}

\begin{claim}\label{claim:charge-k}
For every vertex $v$, the final charge $\mu^*(v)$ is non-positive. 
\end{claim}
\begin{proof}
If $v\in I$, then $v$ has exactly $k$ $J$-neighbors since $G$ is $k$-regular, so the final change of $v$ is zero by {\bf[R1]}.
If $v\in J_i$ for $i\in\{2, \ldots, k-1\}$, then $\mu^*(v)\leq 1-k-(-1)\cdot i-\min\{0,-k+1+i\}=0$ by  {\bf[R1]} and {\bf[R2]}. 
If $v\in J_1$, then $\mu^*(v)\leq 1-k-(-1)-(2-k)=0$ by {\bf[R1]} and {\bf[R3]}. 

Now it remains to check the final charge of a vertex $v$ in $J_k$. 
Note that $v$ always sends $(-1)\cdot k$ to its neighbors, which are all $I$-neighbors, by {\bf[R1]}.
So in order for the final charge of $v$ to be non-positive, it must receive charge at most $-1$ from other vertices. 

\begin{enumerate}[(1)]
\item Suppose $X(v) \cap J_{[1,k-2]} = \emptyset$.

By Claim~\ref{claim-basic2-k}, $X(v)$ consists of $k-1$ distinct vertices in $J_{k-1}$ with 
a common $J_1$-neighbor $u$ such that  $|Y(u)|\le k-3$ and $\cup_{w\in J_{k-1}\cap N_G(u)}Y(w)=\{v\}$. 
Thus $u$ sends $-1$ to $v$ by {\bf[R3-3]}.

\item Suppose $X(v) \cap J_{[1,k-2]} \neq \emptyset$. 

\begin{enumerate}[{(2)-}1:]
    \item  $X(v)$ contains a vertex $v_1\in J_1$.
    
If $|Y(v_1)| \le k-2$, then $v_1$ sends  $-1$ to $v$ by {\bf[R3-2]}, 
so suppose $|Y(v_1)| = k-1$.
If there is a vertex $v_2 \in (X(v)\setminus\{v_1\}) \cap J_{[1,k-2]}$, then $v_2$ sends  at most $-\frac{1}{k-1}$ to $v$ and $v_1$ sends  $-\frac{k-2}{k-1}$ to $v$ by {\bf[R2]} and {\bf[R3]}, so $v$ receives at most $-1$. 
If $(X(v) \setminus \{v_1\}) \cap J_{[1,k-2]}= \emptyset$, then by Claim~\ref{claim-basic2-k}, $X(v)\setminus J_1$ consists of $k-1$ distinct vertices in $J_{k-1}$ with a common $J_1$-neighbor $u$ such that $|Y(u)|\le k-3$ and $\cup_{w\in J_{k-1}\cap N_G(u)}Y(w)=\{v\}$.
Thus $u$ sends $-1$ to $v$ by {\bf[R3-3]}. 

\item $X(v)\cap J_1=\emptyset$. 

Let $i$ be the maximum integer such that $X(v) \cap J_{[1,i]} = \emptyset$. 
Since $X(v) \cap J_{[1,k-2]} \neq \emptyset$, there is a vertex $v_1 \in X(v) \cap J_{i+1}$ where $1\le i\le k-3$.
Moreover, by Claim~\ref{claim-basic-k}(ii),
there are $i+1$ vertices $v_1,\ldots,v_{i+1}\in X(v) \cap J_{i+1}$ such that $N_I(v_1) = \cdots = N_I(v_{i+1})$ since $X(v) \cap J_{[1,i]} = \emptyset$.
Since $G$ is $k$-regular, 
$|Y(v_j)| \le k-i-1$ for all $j \in [i+1]$, so $v$ receives at most
$ \left(-\frac{k-1-(i+1)}{k-i-1}\right)\cdot (i+1) = \frac{i(i+3-k)-k+2}{k-i-1} \le  -1$ from $v_1, \ldots, v_{i+1}$ by {\bf[R2]}.
Note that last inequality holds since $1\le i\le k-3$. 
\end{enumerate}
\end{enumerate}
\end{proof}

\section{Ratio of independent domination and domination for regular graphs}\label{sec:ratio}

In this section, we prove  Theorem~\ref{mainthm:kreg}.
Fix $k \ge 4$, and let $G$ be a connected $k$-regular graph on $n$ vertices that is not  $K_{k,k}$.
For simplicity, let $n_1(H)=|V(H)|-n_0(H)$ for every graph $H$.
We prove that the following statement holds.

\begin{claim}\label{clm:igd}
For a dominating set $D$ of $G$, 
$i(G) \le |D|+(k-3) \cdot n_1(G[D])$.
\end{claim}
\begin{proof}
We use induction on $n_1(G[D])$.
 If $n_1(G[D])=0$, then $D$ is an independent dominating set of $G$, so the statement holds since $i(G) \le |D|$.

Now, assume  $n_1(G[D])>0$.
 Take a vertex $v \in D$ with maximum degree in $G[D]$, and let $\deg_{G[D]}(v)=d$.
 Note that $d\geq 1$, and let $P = \{u \in V(G) \setminus D \mid N_G[u] \cap D = \{v\}\}$.
 Take a maximal independent set $P'$ of $G[P]$, and let $D' = (D \setminus \{v\}) \cup P'$. Note that $D'$ is a dominating set of $G$ by the maximality of $P'$.
Since $G$ is $k$-regular, $|P'|\le k-d$, so $|D'| = |D| + |P'| - 1 \le |D| + k-d-1$.
Note that all vertices in $P'$ are isolated vertices in $G[D']$, and 
therefore $n_1(G[D']) \le n_1(G[D])-1-q$, where $q$ is the number of pendent neighbors of $v$ in $G[D]$. 
If $d=1$, then the neighbor of $v$ in $G[D]$ is also a pendent vertex in $G[D]$, so  $q= 1$. Hence, since $k\ge 4$, 
it holds that $-d+2-q(k-3)\le -d+2-q \le 0$.
By the induction hypothesis, \begin{eqnarray*}
i(G) \le  |D'| + (k-3) \cdot n_1(G[D']) &\le& (|D|+k-d-1) + (k-3)(n_1(G[D])-1-q) \\
&=& |D| + (k-3) \cdot n_1(G[D]) -d+2-q(k-3) \\
&\le& |D| + (k-3) \cdot n_1(G[D]).
\end{eqnarray*} 
\end{proof}

Let $D$ be a minimum dominating set of $G$, so $|D| = \gamma(G)$.
Let $c_k = \frac{k^2-4k+2}{k^2-2k}$. 
We have two cases.

{\bf Case 1}:
Suppose that $n_0(G[D]) \ge c_k \cdot n_1(G[D])$. 
Then \[\gamma(G)=|D|=n_0(G[D])+n_1(G[D]) \ge (c_k+1) \cdot n_1(G[D]), \mbox{ so }
\frac{ n_1(G[D])}{\gamma(G)} \le \frac{1}{c_k+1}.\]
By Claim~\ref{clm:igd}, $i(G) \le |D|+(k-3) \cdot n_1(G[D])$, and therefore,
\begin{eqnarray*}
\frac{i(G)}{\gamma(G)} 
\le \frac{\gamma(G) + (k-3)\cdot n_1(G[D])}{\gamma(G)}
=1+\frac{(k-3) \cdot n_1(G[D])}{\gamma(G)}  \le 1+ \frac{k-3}{c_k+1}
=\frac{k^3-3k^2+2}{2k^2-6k+2}.
\end{eqnarray*}

{\bf Case 2}:
Suppose that $n_0(G[D]) \le c_k \cdot n_1(G[D])$. 
Note that $n \le (k+1)\cdot n_0(G[D]) + k \cdot n_1(G[D])$ since $D$ is a dominating set, and each vertex in $n_0(G[D])$ (resp. $n_1(G[D])$) is adjacent to at most $k$ (resp. $k-1)$ vertices not in $D$. 
Thus
\[  \gamma(G)= n_0(G[D])+n_1(G[D]) \ge \frac{n-k \cdot n_1(G[D])}{k+1}+n_1(G[D])=\frac{n+n_1(G[D]) }{k+1} .\]
Since $n_0(G[D]) \le c_k \cdot n_1(G[D])$, we obtain $n \le ((k+1)c_k +k) \cdot n_1(G[D])$, and therefore,
\[  \gamma(G)\ge \frac{n}{k+1}+\frac{n}{(k+1)((k+1)c_k+k)} 
=\frac{n(2k^2-6k+2)}{(2k-1)(k^2-2k-2)}. \]
By Theorem~\ref{mainthm:kreg}, $i(G) \le \frac{n(k-1)}{2k-1}$, and therefore,
\begin{eqnarray*}
\frac{i(G)}{\gamma(G)} 
&\le&  
\frac{k-1}{2k-1} \cdot \frac{(2k-1)(k^2-2k-2)}{2k^2-6k+2}
=\frac{k^3-3k^2+2}{2k^2-6k+2}.
\end{eqnarray*}

\section{Independent domination for graphs with maximum degree $4$}\label{sec:idset4}

In this section, we prove  Theorem~\ref{thm:idset4}.
We actually prove the following slightly stronger statement, whose direct consequence is  Theorem~\ref{thm:idset4}.
 
\begin{theorem}\label{thm:weight}
If $G $ is a graph with maximum degree at most $4$, then $9i(G) \le 5|V(G)|+4n_0(G)$.
\end{theorem}
  
\begin{proof}
Suppose to the contrary that a graph $G$ is a minimum counterexample to Theorem~\ref{thm:weight} with respect to the number of vertices.
In particular, $9i(G)> 5|V(G)|+4n_{0}(G)$ and $9i(H) \le 5|V(H)|+4n_{0}(H)$ for every proper subgraph $H$ of $G$. 

Note that $|V(G)|\ge 2$ since the theorem holds for the graph with a single vertex. 
If $G$ is the disjoint union of two graphs $G_1$ and $G_2$, then  by the minimality of $G$, we obtain
\[ 9i(G) = 9i(G_1)+9i(G_2) \le (5|V(G_1)|+4n_0(G_1))+(5|V(G_2)|+4n_0(G_2)) = 5|V(G)|+4n_0(G), \]
which is a contradiction. 
Thus $G$ is connected, so $n_0(G)=0$, therefore, $5|V(G)|< 9i(G)$.
 
For simplicity, denote the set of isolated vertices of $G-N_G[v]$ by $I_G(v)$. 
By counting the number of edges between $N_G[v]$ and $G-N_G[v]$, we know $|I_G(v)|\le \frac{3\deg_G(v)}{\delta(G)}$, since $G$ is a connected graph with maximum degree at most $4$.  
Adding $v$ to an  independent dominating set of $G-N_G[v]$ is an independent dominating set of $G$.
Thus, by the minimality of $G$,
\begin{eqnarray*}
5|V(G)|< 9i(G)\le  9  + 9i(G-N_G[v]) &\le& 9  + 5|V(G-N_G[v])|+4n_0(G-N_G[v])\\
&=&9+5|V(G)|-5(\deg_G(v)+1)+4|I_G(v)|, \end{eqnarray*}
so 
$4|I_G(v)| \ge 5\deg_G(v)-3$ since each term is an integer. 
Hence, it holds that 
\begin{eqnarray}\label{claim:basic:I}
\forall v\in V(G), \quad  \quad 
4\cdot \frac{3\deg_G(v)}{\delta(G)} \ge 4|I_G(v)| \ge 5\deg_G(v)-3.
\end{eqnarray}
If $\delta(G)\ge 3$,
then \eqref{claim:basic:I} implies that $\deg_G(v)=3$ for every vertex $v$ so $G$ is $3$-regular.
Now, \eqref{claim:basic:I} again implies that $|I_G(v)|=3$ for every vertex $v$, which is impossible in a $3$-regular graph. 
If $\delta(G)=2$, then by considering a $2$-vertex $v$,  \eqref{claim:basic:I} implies that $|I_G(v)|\ge 2$.
This further implies that $v$ has a neighbor $x$ such that $\deg_G(x)\ge 3$ and $|I_G(x)|\le 1$, which is a contradiction to \eqref{claim:basic:I}. 
Hence, $\delta(G)=1$. 

\begin{claim} \label{claim:delta2} 
If  a vertex $v$ has a pendent neighbor, then $v$ is a $4$-vertex with exactly two pendent neighbors.
\end{claim}
\begin{proof}
Let $v$ be a vertex with the maximum number of pendent neighbors such that $\deg_G(v)=d_1+d_2$ where $d_1$ denotes the number of pendent neighbors of $v$.
Since each neighbor of $v$ has at most $d_1$ pendent neighbors, by counting the number of edges between $N_G[v]$ and $G-N_G[v]$, we know $|I_G(v)|\le d_1d_2+ \frac{ 3d_2-d_1d_2}{2}$. 
By \eqref{claim:basic:I}, \begin{eqnarray}\label{eq:d1d2}
&& 4d_1d_2+ 2 ({ 3d_2-d_1d_2}) \ge 4|I_G(v)| \ge 5d_1+5d_2-3.
\end{eqnarray}
Since $d_1+d_2\le 4$, 
this implies  $d_1\le 2$. 
Thus, $v$ has at most two pendent neighbors.

Note that by \eqref{claim:basic:I}, for a pendent vertex $w$, $|I_G(w)|\ge 1$, which implies that every vertex with a pendent neighbor has at least two pendent neighbors.
Hence, $v$ has exactly two pendent neighbors. 

Moreover, $d_1=2$ in \eqref{eq:d1d2} results in $d_2\geq 2$, so we conclude that $v$ must be a $4$-vertex. 
\end{proof}

Consider a vertex $v$ with a pendent neighbor. 
By Claim~\ref{claim:delta2},  $v$ is a $4$-vertex with exactly two pendent neighbors.
By \eqref{claim:basic:I}, $|I_G(v)|\ge 5$. 
Since every vertex has at most two pendent neighbors,  $|I_G(v)|\le 5$. This further implies that $G$ is the graph with 10 vertices obtained from a $4$-cycle $v_1v_2v_3v_4v_1$ by attaching exactly two pendent neighbors to each of $v_1$, $v_2$, and $v_3$. One may check easily that this is not a counterexample to our theorem. 
 \end{proof}

\section*{Acknowledgements}
Eun-Kyung Cho was supported by Basic Science Research Program through the National Research Foundation of Korea (NRF) funded by the Ministry of Education (NRF-2020R1I1A1A0105858711).
Ilkyoo Choi was supported by the Basic Science Research Program through the National Research Foundation of Korea (NRF) funded by the Ministry of Education (NRF-2018R1D1A1B07043049), and also by the Hankuk University of Foreign Studies Research Fund.
Boram Park was supported by Basic Science Research Program through the National Research Foundation of Korea (NRF) funded by the Ministry of Science, ICT and Future Planning (NRF-2018R1C1B6003577).

\bibliographystyle{abbrv}
\bibliography{ref}

\begin{thebibliography}{10}

\bibitem{book3}
{\em Topics in domination in graphs}, volume~64 of {\em Developments in
  Mathematics}.
\newblock Springer, Cham, 2020.

\bibitem{abrishami2018independent}
G.~Abrishami and M.~A. Henning.
\newblock Independent domination in subcubic graphs of girth at least six.
\newblock {\em Discrete Math.}, 341(1):155--164, 2018.

\bibitem{akbari2021independent}
A.~Akbari, S.~Akbari, A.~Doosthosseini, Z.~Hadizadeh, M.~A. Henning, and
  A.~Naraghi.
\newblock Independent domination in subcubic graphs.
\newblock {\em J. Comb. Optim.}, Online published.

\bibitem{babikir2020domination}
A.~Babikir and M.~A. Henning.
\newblock Domination versus independent domination in graphs of small
  regularity.
\newblock {\em Discrete Math.}, 343(7):111727, 2020.

\bibitem{berge1962theory}
C.~Berge.
\newblock {\em The theory of graphs and its applications}.
\newblock Methuen \& Co. Ltd., London; John Wiley \& Sons Inc., New York, 1962.
\newblock Translated by Alison Doig.

\bibitem{dorbec2015independent}
P.~Dorbec, M.~A. Henning, M.~Montassier, and J.~Southey.
\newblock Independent domination in cubic graphs.
\newblock {\em J. Graph Theory}, 80(4):329--349, 2015.

\bibitem{duckworth2006independent}
W.~Duckworth and N.~C. Wormald.
\newblock On the independent domination number of random regular graphs.
\newblock {\em Combin. Probab. Comput.}, 15(4):513--522, 2006.

\bibitem{furuya2014ratio}
M.~Furuya, K.~Ozeki, and A.~Sasaki.
\newblock On the ratio of the domination number and the independent domination
  number in graphs.
\newblock {\em Discrete Appl. Math.}, 178:157--159, 2014.

\bibitem{goddard2013independent}
W.~Goddard and M.~A. Henning.
\newblock Independent domination in graphs: a survey and recent results.
\newblock {\em Discrete Math.}, 313(7):839--854, 2013.

\bibitem{goddard2012independent}
W.~Goddard, M.~A. Henning, J.~Lyle, and J.~Southey.
\newblock On the independent domination number of regular graphs.
\newblock {\em Ann. Comb.}, 16(4):719--732, 2012.

\bibitem{book4}
T.~W. Haynes, S.~T. Hedetniemi, and M.~A. Henning, editors.
\newblock {\em Structures of domination in graphs}, volume~66 of {\em
  Developments in Mathematics}.
\newblock Springer, Cham, 2021.

\bibitem{book2}
T.~W. Haynes, S.~T. Hedetniemi, and P.~J. Slater, editors.
\newblock {\em Domination in graphs}, volume 209 of {\em Monographs and
  Textbooks in Pure and Applied Mathematics}.
\newblock Marcel Dekker, Inc., New York, 1998.
\newblock Advanced topics.

\bibitem{book1}
T.~W. Haynes, S.~T. Hedetniemi, and P.~J. Slater.
\newblock {\em Fundamentals of domination in graphs}, volume 208 of {\em
  Monographs and Textbooks in Pure and Applied Mathematics}.
\newblock Marcel Dekker, Inc., New York, 1998.

\bibitem{knor2020domination}
M.~Knor, R.~{\v{S}}krekovski, and A.~Tepeh.
\newblock Domination versus independent domination in regular graphs.
\newblock {\em arXiv preprint arXiv:2010.13467}, 2020.

\bibitem{lam1999independent}
P.~C.~B. Lam, W.~C. Shiu, and L.~Sun.
\newblock On independent domination number of regular graphs.
\newblock {\em Discrete Math.}, 202(1-3):135--144, 1999.

\bibitem{suil2016cubic}
S.~O and D.~B. West.
\newblock Cubic graphs with large ratio of independent domination number to
  domination number.
\newblock {\em Graphs Combin.}, 32(2):773--776, 2016.

\bibitem{ore1962theory}
O.~Ore.
\newblock {\em Theory of graphs}.
\newblock American Mathematical Society Colloquium Publications, Vol. XXXVIII.
  American Mathematical Society, Providence, R.I., 1962.

\bibitem{rad2013note}
N.~J. Rad and L.~Volkmann.
\newblock A note on the independent domination number in graphs.
\newblock {\em Discrete Appl. Math.}, 161(18):3087--3089, 2013.

\bibitem{rosenfeld1964independent}
M.~Rosenfeld.
\newblock Independent sets in regular graphs.
\newblock {\em Israel J. Math.}, 2:262--272, 1964.

\bibitem{southey2013domination}
J.~Southey and M.~A. Henning.
\newblock Domination versus independent domination in cubic graphs.
\newblock {\em Discrete Math.}, 313(11):1212--1220, 2013.

\end{thebibliography}

\end{document}